\documentclass[12pt,reqno,tbtags]{amsart}
\usepackage{fourier}
\usepackage{nicefrac}

\usepackage{amsmath,amsfonts,latexsym,amsthm,amstext,enumerate,color}
\usepackage[numeric,initials,nobysame]{amsrefs}

\numberwithin{equation}{section}

\newtheorem{theorem}{Theorem}[section]
\newtheorem*{theorem*}{Theorem}
\newtheorem{lemma}[theorem]{Lemma}
\newtheorem{claim}[theorem]{Claim}

\newtheorem*{observation*}{Observation}

\theoremstyle{definition}{

\newtheorem*{remark*}{Remark}
}

\newcommand{\N}{\mathbb N}

\newcommand{\eps}{\varepsilon}
\newcommand{\ra}{\rightarrow}

\renewcommand{\and}{\hbox{ {\rm and} }}

\newcommand{\E}{\mathbb{E}}
\renewcommand{\P}{\mathbf{P}}

\DeclareMathOperator{\var}{Var}
\DeclareMathOperator{\cov}{Cov}

\newcommand{\cF}{\mathcal{F}}

\newcommand{\tcov}{\tau_{\rm cov}}
\newcommand{\thit}{\tau_{\rm hit}}

\renewcommand{\epsilon}{\varepsilon}

\newcommand{\be}{\begin{eqnarray}}
\newcommand{\ee}{\end{eqnarray}}

\date{}

\begin{document}
\title{Linear cover time is exponentially unlikely}

\author{Itai Benjamini}
\address{Itai Benjamini\hfill\break
Weizmann Institute of Science}
\email{itai.benjamini@weizmann.ac.il}
\urladdr{}

\author{Ori Gurel-Gurevich}
\address{Ori Gurel-Gurevich\hfill\break
University of British Columbia}
\email{origurel@math.ubc.ca}
\urladdr{}

\author{Ben Morris}
\address{Ben Morris\hfill\break
University of California, Davis}
\email{morris@math.ucdavis.edu}
\urladdr{}

\begin{abstract}
We show that the probability that a simple random walk covers a finite, bounded degree graph in linear time is exponentially small.

More precisely, for every $D$ and $C$, there exists $\alpha=\alpha(D,C)>0$ such that for any graph $G$, with $n$ vertices and maximal degree $D$, the probability that a simple random walk, started anywhere in $G$, will visit every vertex of $G$ in its first $C n$ steps is at most $e^{-\alpha n}$.

We conjecture that the same holds for $\alpha$ that does not depend on $D$, provided that the graph $G$ is simple.
\end{abstract}

\maketitle




\section{Introduction}

Let $G=(V,E)$ be a finite connected graph, let $\{X_t\}_{t=0}^\infty$ be a simple random walk on $G$ started at $X_0=v$. Let $\tcov$ be the \emph{cover time} of the walk, i.e. the first time $t$ such that for every $v\in G$ there is $s\le t$ such that $X_s=v$. Our main result is:

\begin{theorem} \label{mainthm}
For every $D$ and $C$, there exists $\alpha=\alpha(D,C)>0$ such that for any graph $G$, with $n$ vertices and maximal degree $D$, and every starting vertex $v\in V$ we have
$$\P_v(\tcov\le C n) \le e^{-\alpha n} \, .$$
\end{theorem}

In certain special cases, the result follows from a direct application
of Hoeffding's inequality. For example, if the graph is a path of length $n$
then the probability to hit the end of the path within $Cn$ steps is exponentially small. However, this approach fails in general since more typically
there is a fixed probability to have hit any specific vertex by time $Cn$.

A naive approach to this problem would be to consider the Doob martingale of some related random variable. Natural choices include either the cover time itself or the number of uncovered vertices. However, these martingales could have large differences. For example when considering a simple random walk on a complete binary tree of height $h$, if the walk has already covered half of the tree and is now at the root, the next step would cause a very large change in the value of either of these martingales.

The proof of Theorem \ref{mainthm} relies heavily on the following fact: The expected number of visits to a vertex $v$ before covering $B_v(r)$ (the ball of radius $r$ around $v$) grows to infinity with $r$, \textbf{even when we allow the walk to behave arbitrarily outside of $B_v(r)$}. To make this more precise, let us make some definitions.

A stochastic process $X_t$ on the vertices of $G$ is said to be a \emph{random walk} if $X_{t+1}$ is a neighbor of $X_t$, almost surely. For a subset of the vertices $S\subset V$, a random walk in $S$-\emph{simple} if the distribution of $X_{t+1}$ given the history $X_0,\ldots,X_t$ is uniform on the neighbors of $X_t$ whenever $X_t \in S$.

For $X$ a random walk on $G$ and $S$ a subset of vertices let $\tcov^*(S)$ be the first time $t$ such that $X_t \not \in S$ and for every $v\in S$ there is $s\le t$ such that $X_s=v$. Let $\ell^v_t=|\{s < t \mid X_s=v\}|$ be the number of visits to $v$ until time $t$.

\begin{lemma} \label{mainlem}
For every $D$ and $C$, there exists $r=r(D,C)$, such that if $G$ is a graph of maximal degree at most $D$ and $v$ is a vertex of $G$ such that $B_v(r)\ne V$, then any $B_v(r)$-simple random walk, started outside $B_v(r)$ satisfies
$$\E(\ell^v_{\tcov^*(B_v(r))})\ge C \, .$$
\end{lemma}

The proof of Theorem \ref{mainthm} then proceeds by constructing a certain submartingale (which is reminiscent of the Doob martingale), which bounds the cover time from below, has expectation $2Cn$ and has bounded differences. Then by Hoeffding's bounds, the value of this submartingale at time $Cn$ is exponentially unlikely to be less then $Cn$, which means that the walk hasn't covered the graph by this time.

Lemma \ref{mainlem} is of interest in itself. For example, a direct consequence is the well-known fact that the expected cover time of bounded degree graphs grows superlinearly in the number of vertices (see subsection \ref{related}). A more subtle implication is that for this result to hold one only needs the random walk to be simple in the vicinity of some constant fraction of the vertices. In particular, the cover time of random walk on a bounded degree graph which is simple in all but a sublinear number of vertice is also superlinear. In fact, our main Theorem applies to these kind of random walks as well.

An interesting open question is to determine the right quantitative version of \ref{mainlem}. One can obtain an exponential lower bounded for $r$ in terms of $C$ (and fixed $D$) by considering a simple random walk on a $d$-dimensional torus, for $d\ge 3$. The question is whether the power to change the behavior of the walk outside of $B_v(r)$ can reduce significantly the expected number of visits to $v$ before covering.


\subsection{Related Works} \label{related}

The cover time of a simple random walk on graphs is a fairly natural concept which has been studied extensively in the past 30 years. Almost all results about the cover time are about its expectation. The most important lower bound is that of Feige \cite{Feige} who proved that $\E_u(\tcov)\ge (1-o(1))n \log n$ for any simple graph on $n$ vertices and any starting vertex $u$. This implies that the probability to cover the graph in $Cn$ steps cannot be more than $O(C/\log n)$ uniformly for all vertices.

The only concentration-type result the authors are aware of is that of Aldous \cite{Aldous} who proved that if $\max_{u,v}\E_u(\thit^v) \ll \max_u \E_u(\tcov)$ (where $\thit^v$ is the first time the walk visits $v$) then for any starting vertex $u$ we have $\tcov / \E_u(\tcov) \to 1$ in distribution. Notice that our main result applies for any bounded degree graph, even if the cover time is not concentrated around its mean.

The interested reader is referred to \cite{LPW, AF} for further information about the cover time. More information about the importance of cover times in Computer Science can be found in \cite{Z}.

\section{Proof of the main Theorem}

Given a graph $G=(V,E)$, a vertex $v \in V$ and $r\in \N$ let $A_v(r)$ be the annulus of radius $r$ around $v$ and assume that $A_v(r)\ne \emptyset$.  (For the convenience of the reader, we have included a legend of notation on the last page.) Given a walk $X_t$ on $G$ let $\cF_t=\sigma(X_0,\ldots,X_t)$ and let $\ell^v(r)=\ell^v_{\tcov^*(B_v(r))}$ be the number of visits to $v$ before covering and exiting $B_v(r)$ (or $\infty$ if the walk never covers $B_v(r)$). Define
\be \label{submardef} L^v_t(r) = \inf \, \E(\ell^v(r)(Y) \mid \cF_t) \ee
where the infimum is taken over all $B_v(r)$-simple random walks $Y$ that agree with $X$ in the first $t$ steps (i.e. $\P(Y_0=X_0,\ldots,Y_t=X_t)=1$). The stochastic process $\{L^v_t(r): t \geq 0\}$
is adapted to the filtration $\cF_t$ and is somewhat similar to the Doob martingale. However, here we take expectation with respect to a different process than the random walk itself.

The next few Lemmas show that $L^v_t(r)$ is, in fact, a submartingale with bounded differences and that it does not change its value when the walk is outside of $B_v(r)$.

\begin{lemma} \label{submartingale}
$L^v_t(r)$ is a sub-martingale.
\end{lemma}
\begin{proof}
\begin{eqnarray*}
\E(L^v_{t+1}(r) \mid \cF_t) &=& d_{X_t}^{-1} \sum_{u\sim X_t} \E(L^v_{t+1}(r) \mid \cF_t,X_{t+1}=u)\\
 &=& d_{X_t}^{-1} \sum_{u\sim X_t} \inf\, \E(\ell^v(r)(Y) \mid \cF_t, X_{t+1}=u)
\end{eqnarray*}
where for each summand the infimum is taken over all $B_v(r)$-simple random walks which agree with $X$ in the first $t+1$ steps. Given a vector $\{Y^u\}_{u\sim X_t}$ of such random walks we can combine them into a single such random walk $Y$ in the following way: $Y_s=X_s$ for $s\le t+1$ and $Y_s=Y^u_s$ for $s>t+1$ if $X_{t+1} = u$. Obviously, $\E(L^v_{t+1}(r) \mid \cF_t,X_{t+1}=u)$ is the same under $Y^u$ and under $Y$. Hence
\begin{equation}
\label{equal}
 d_{X_t}^{-1} \sum_{u\sim X_t} \inf\, \E(\ell^v(r)(Y) \mid \cF_t, X_{t+1}=u) \geq \inf \, \E(\ell^v(r)(Y) \mid \cF_t)
\end{equation}
where the infimum is now taken over all $B_v(r)$-simple random walks which agree with $X$ in the first $t+1$ steps.
(In fact we have equality in equation (\ref{equal}), but we don't need this.)
In comparison, in the definition of $L^v_t(r)$ we have the same expectation but the infimum is taken over all $B_v(r)$-simple random walks which agree with $X$ in the first $t$ steps. This latter set contains the former, hence
$$\E(L^v_{t+1}(r) \mid \cF_t) \ge L^v_t(r) \, .$$
%
%
%
\end{proof}

\begin{lemma} \label{bounded_range}
If $X_t\not\in B_v(r)$ and $X_{t+1}\not\in B_v(r)$ then $L^v_{t+1}(r)=L^v_t(r)$.
\end{lemma}
\begin{proof}
Since the infimum in the definition of $L^v_t(r)$ includes all the $B_v(r)$-simple random walks $Y$ where $Y_{t+1}=X_{t+1}$ with probability 1, we see that $L^v_{t+1}(r)\ge L^v_t(r)$. Similarly, if we have $X_{t+2}=X_t$ then $L^v_{t+2}(r)\ge L^v_{t+1}(r)$. However, since $L^v_t(r)$ only depends on $X_t$ and on which vertices were visited in $B_v(r)$ and on $\ell^v_t$ and none of these changes between time $t$ and $t+2$ if $X_{t+2}=X_t$, we get that $L^v_t(r)=L^v_{t+2}(r) \ge L^v_{t+1}(r) \ge L^v_t(r)$.
\end{proof}

In fact, when inside $B_v(r)$, this process is a martingale and when traversing an edge outside of $B_v(r)$ its value doesn't change, so the only times when $L^v_t(r)$ exhibits its ``sub''-ness is when taking a step from the outside to the inside of $B_v(r)$.

\begin{lemma} \label{bounded_differences}
There exists $M=M(D,r)$, such that $|L^v_{t+1}(r)-L^v_t(r)|\le M$.
\end{lemma}
\begin{proof}
Consider $L^v_t(r)-\ell^v_t$. This is the infimum of the expected number of visits to $v$ between times $t$ and $\tcov^*(B_v(r))$ where the infimum is with respect to any $B_v(r)$-simple random walk that agrees with $X$ in the first $t$ steps. This number is nonnegative and  bounded above by the expectation when we take the walk $X$ itself. This  is at most $D^{2D^{r+1}} + 2D^{r+1}$, since after every visit to $v$ there is a probability of at least $D^{-2D^{r+1}}$ that $X$ now performs a depth first search of $B_v(r)$, and during such a search the walk may visit $v$ no more than $2D^{r+1}$ times. Since $|\ell^v_{t+1}-\ell^v_t|\le 1$ we get that
$$|L^v_{t+1}(r)-L^v_t(r)| \le D^{2D^{r+1}} + 2D^{r+1} +1 \, .$$
\end{proof}

Now we can turn to the proof of the main result.

\begin{proof}[Proof of Theorem \ref{mainthm}]
Given $D$ and $C$, let $r=r(D,4C)$, as given by Lemma \ref{mainlem}. If $G=(V,E)$ is a connected graph with maximal degree at most $D$, and $n=|V|>D^{r+1}$ then for every $v\in V$ we have $A_v(r)\ne \emptyset$. Hence, we can define $L^v_t(r)$ and we have $L^v_0(r)\ge 4C$ for all $v\in V \setminus B_{X_0}(r)$.

Consider the sum
$$L_t=\sum_{v\in G \setminus B_{X_0}(r)} L^v_t(r) \, .$$

By Lemma \ref{submartingale} we know that $L_t$ is a sub-martingale too, since all of the $L^v_t(r)$ are adapted to the same filtration. Combining Lemmas \ref{bounded_differences} and \ref{bounded_range} shows that $|L_{t+1}-L_t|\le M$, provided we incorporate a factor $D^{r+2}$
into the constant $M=M(D,r)$ from Lemma \ref{bounded_differences}.
We now have
$$L_0 \ge 4C \big|V \setminus B_{X_0}(r)\big|\ge 3Cn \, ,$$
for sufficiently large $n$.

We can now apply the Hoeffding-Azuma inequality to get
$$\P(L_{t} \le 2Cn) \le e^{-n^2/2tM}$$
for any $t$.

Substituting $t=2Cn$ we get
\begin{equation}
\label{star}
\P(L_{2Cn} \le 2Cn) \le e^{-n/4CM} \, .
\end{equation}

Let $\tcov^*$ be the first time $t>\tcov$ such that $X_t \not \in B_{X_{\tcov}}(2r)$.
Note that
$\tcov^* \ge \tcov^*(B_v(r))$ for all $v\in V$.
Note also  that if $t\ge \tcov^*(B_v(r))$ then $\ell^v_t \ge \ell^v_{\tcov^*(B_v(r))} =
L^v_t$ for all $v \in V$, and summing this inequality over $v$ gives
\begin{eqnarray*}
t &=& \sum_{v \in V} \ell^v_t \\
&\ge& \sum_{v  \in V} L^v_t = L_t\, .
\end{eqnarray*}
Thus if $L_t > t$ then we must have $\tcov^*(B_v(r)) > t$ for some $v\in V$ and hence $\tcov^* > t$ as well. Thus
$\P(\tcov^* \leq t)  \leq \P(L_t \leq t)$. Substituting $t = 2Cn$ gives
\begin{eqnarray*}
\P( \tcov^* \leq 2Cn) &\leq& \P( L_{2Cn}  \leq 2Cn) \\
&\leq& e^{-n/4CM},
\end{eqnarray*}
by equation (\ref{star}).
Finally we note that $\P(\tcov^* - \tcov \ge t)$ decays exponentially fast, at a rate depending only on
 $D$ and $r$, regardless of the history until time $\tcov$. Hence,
$$\P(\tcov<Cn) \le \P(\tcov^* \leq  2Cn) + \P(\tcov^* - \tcov \ge Cn) \le e^{-\alpha n} \, ,$$
for a constant $\alpha$ that depends only on $D$ and $r$ which in turn depends only on $D$ and $C$.
\end{proof}

%
%
%
%


\section{Proof of the main Lemma}

%

Define
$$\phi(r)=\min \E(\ell^v(r)) / d_v$$
where the minimum is take over all connected graphs $G=(V,E)$ of maximal degree at most $D$ and vertices $v \in V$ such that $A_v(r) \ne \emptyset$ and over all $B_v(R)$-simple random walks started outside of $B_v(r)$. Then one may restate the main Lemma as
$$\lim_{r\ra \infty} \phi(r) = \infty \, .$$
We will prove this fact by induction on the value of $\phi(r)$. More precisely, we will show that if $\phi(r)=K$ then there is some $R>r$ such that $\phi(R)\ge K+ e^{-3 K d_v -4}$. Obviously, this is enough, as iterations of $K \mapsto K+e^{-3 K d_v -4}$ tend to infinity.

For a set of vertices $S\subset V$ write $B_S(r)=\cup_{v\in S} B_v(r)$. The following is a weaker, but more general version of Theorem \ref{mainthm}, showing that the weighted sum of the number of visits to a set $S$ of vertices is unlikely to be small for $B_S(r)$-simple random walk.

\begin{lemma} \label{weak_concentration}
Fix $r \in \N$ and let $K=\phi(r)$. For any $\eps>0$ there is some $a=a(r,\eps)>0$ such that if $G=(V,E)$ is a connected graph of maximal degree at most $D$ and $\{a_v\}_{v\in V}$ a probability distribution on some $S\subset V$ with $\max_{v \in S} a_v \le a$ then for any $B_S(r)$-simple random walk started outside $B_S(r)$ we have
$$\P\big(\sum_{v\in S} a_v \ell^v(r) < K \sum_{v\in S} a_v d_v -\eps\big) < \eps \, .$$
\end{lemma}
\begin{proof}
For any submartingale $L_t$ one can construct a martingale $M_t$ such that
\begin{enumerate}
\item $M_0=L_0$

\item $M_t\le L_t$

\item If $L_{t+1}=L_t$ then $M_{t+1}=M_t$

\item If the differences of $L_t$ are bounded by $L$ then the differences of $M_t$ are bounded by $2L$
\end{enumerate}

Now, apply this to $L^v_t(r)$ to get the martingales $M^v_t$, and let
$$M^v=\lim_{t\ra \infty} M^v_t \le \lim_{t\ra \infty} L^v_t(r) = \ell^v(r) \, .$$

It now follows that $M^v$ and $M^u$ are uncorrelated when the distance between $v$ and $u$ is more than $2r$. This is since $M^v=\sum_{t=0}^\infty M^v_{t+1}(r) - M^v_t(r)$ and we have $(M^v_{t+1}-M^v_t)(M^u_{t+1}-M^u_t)=0$ by Lemma \ref{bounded_range} and property 3 above and for $s\ne t$ we have $\E((M^v_{t+1}-M^v_t)(M^u_{s+1}-M^u_s))=0$ because these are martingales. Also, the variance of each $M^v$ is bounded by the second moment of $\ell^v(r)$ which is bounded by some function of $D$ and $r$ only, since $\ell^v(r)$ has exponential tails with parameter depending only on $D$ and $R$
(see Lemma \ref{bounded_differences}). Let $N=N(D,r)$ be such a bound for $\var(M^v)$.

Now let
$$M=\sum_{v\in S} a_v M^v \le \sum_{v\in S} a_v \ell^v(r) \, .$$

We bound $\var(M)$ by
\begin{eqnarray*}
\var(M) &=& \sum_{v\in S} \sum_{u\in S} a_v a_u \cov(M^v,M^u)\\
&\le& \sum_{v\in S} \sum_{u in B_v(r)} a_v a_u \cov(M^v,M^u)\\
&\le& \sum_{v\in S} a_v \sum_{u in B_v(r)} a_u \sqrt{\var(M^v)\var(M^u)}\\
&\le& \sum_{v\in S} a_v a D^{r+1} N\\
&\le& a D^{r+1} N \, .
\end{eqnarray*}
and the Lemma holds by choosing $a$ small enough and applying Chebyshev's inequality
\end{proof}

Let $\tau_v(r)$ be the positive hitting time of $A_v(r)$ and for $w\in B_v(r)$ let $a^v_w(r)=\E_w(\ell^v_{\tau_v(r)})$, where the expectation is with respect to a simple random walk. Obviously, this expectation is the same for any $B_v(r)$-simple random walk.

\begin{lemma} \label{reversibility}
Let $G=(V,E)$ be a finite graph, $v\in V$ a vertex and $r\in \N$ such that $A_v(r)\ne \emptyset$. Then
$$\sum_{w\in A_v(r)} d_w a^v_w(r) = d_v \, .$$
\end{lemma}
\begin{proof}
$a^v_w(r)$ is equal to the sum of the probabilities of all paths which start at $w$ and end at $v$ and do not return to $A_v(r)$. For each of these paths, the probability that a simple random walk would traverse it is exactly $d_v/d_w$ times the probability of traversing it in the reverse direction. Hence,
$$\sum_{w\in A_v(r)} d_w a^v_w(r) = d_v \sum_{w\in A_v(r)} \P_v (X_{\tau_v(r)}=w) = d_v$$
where the last equality follows since the walk hits exactly one vertex of $A_v(r)$.
\end{proof}



Let $m^v(r)=\max_{w \in A_v(r)} a^v_w(r)$.

\begin{lemma} \label{small_expectations}
Given a graph $G=(V,E)$ with maximal degree at most $D$ and a vertex $v\in V$ and $r \in \N$ such that $B_v(r) \ne V$, there is $r' \le r$ such that
$$m^v(r') \le \sqrt{\frac{d_v a^v_v(r+1)}{r}}$$
for any $B_v(r)$-simple random walk.
\end{lemma}
\begin{proof}
As in Lemma \ref{reversibility} we have
$$d_w a^v_w(r') = d_v \P_v (X_{\tau_v(r')}=w)$$
for every $w \in A_v(r')$ when $r' \le r$.

One may bound $a^v_v(r)$ by considering, for every $r' \le r$ all the paths which start at $v$, hit $A_v(r')$ at some specified vertex $w$ and then hit $v$ again before returning to $A_v(r')$. For distinct $r'$'s or distinct $w$ in the same $A_v(r')$ these are disjoint sets. This summation yields
\begin{eqnarray*}
a^v_v(r+1) &\ge& \sum_{r' \le r} \sum_{w \in A_v(r')} \P_v (X_{\tau_v(r')}=w) a^v_w(r)\\
&=&\sum_{r' \le r} \sum_{w \in A_v(r')} (a^v_w(r'))^2 d_w /d_v\\
&\ge& \sum_{r' \le r} (m^v(r'))^2 /d_v = r (m^v(r))^2 / d_v
\end{eqnarray*}
where the middle equality follows by reversibility.
\end{proof}

We will also need the following useful Lemma.

\begin{lemma} \label{bit_process}
Let $x_i$ be a stochastic process on $\{0,1\}$, adapted to the filtration $\cF_i$ and let $p_i=\E(x_i \mid \cF_{i-1})$. If $p_i \le \frac12$ a.s. for all $i$, and $\tau$ is a stopping time such that $\sum_{i=1}^\tau p_i \le K$ a.s. then
$$\P(\forall_{i\le \tau} x_i=0) \ge e^{-3 K} \, .$$
\end{lemma}
\begin{proof}
Define $M_i=\prod_{j\le i} (1-p_j)^{-1}$ if $\forall_{j\le i} x_i=0$ and $M_i=0$ otherwise. It is easily checked that $M_i$ is a martingale adapted to $\cF_i$ and $M_0=1$. Since $p_j \le \frac12$ for all $j$ we have $1-p_j \ge e^{-3 p_j}$ so $\prod_{j\le i} (1-p_j)^{-1} \le e^{3 \sum_{j\le i} p_j}$. Since $\sum_{i=1}^\tau p_i \le K$, by the optional stopping Theorem we have $\P(\forall_{i\le \tau} x_i=0) \ge e^{-3 K}$.
\end{proof}


Now we are ready to prove the main Lemma. Very roughly, we show that for some radius $R'$, by the time we cover $A_v(R')$, we visit $v$ almost $K d_v$ times in expectation and there is a non-negligible probability that we haven't visited $v$ at all, in which case we will visit $v$ at least once before covering, thus increasing the expected number of visits to $v$ before covering by this probability.

\begin{proof}[Proof of Lemma \ref{mainlem}]
Let $r$ be such that $K=\phi(r)$. Fix some $\eps$ to be chosen later and let $a=a(r,\eps)$ from Lemma \ref{weak_concentration}. Let $R= D (K+e^{-3 K d_v -4}) /a^2$. We claim that $\phi(R+r)\ge K+e^{-3 K d_v -4}$. This is enough to show that $\lim_{r\ra\infty}\phi(r)=\infty$.

Let $G=(V,E)$ be a graph with maximal degree at most $D$ and let $v\in V$ a vertex such that $A_v(R)\ne \emptyset$. We want to show that for any $B_v(R)$-simple random walk started outside $B_v(R)$ we have $\E(\ell^v(R))\ge K+e^{-3 K d_v -4}$. If $a^v_v(R) \ge K+e^{-3 K d_v -4}$ then we are done (recall that $a^v_w(R)$ is the expected number of visits to $v$ before hitting $A_v(R)$ for a simple random walk started at $w$). Hence, from now on we assume that
\be \label{bounded_resistance} a^v_v(R) \ge K+e^{-3 K d_v -4} \, .\ee

In this case, by Lemma \ref{small_expectations} there is $R' \le R$ such that for all $w\in A_v(R')$ we have
$$a^v_w(R') \le \sqrt{\frac{d_v a^v_v(R)}{R}} \le a \, .$$

Let $t_i$ enumerate the times the walk is in $A^v(R')$ and define
$$b_i=\sum_{j=0}^i a^v_{X_{t_i}}(R')$$
and
$$c_i=\ell^v_{t_{i+1}} \, .$$

\begin{claim} \label{martingale}
$c_i - b_i$ is a martingale (adapted to the filtration $\cF_{t_{i+1}}$).
\end{claim}
\begin{proof}
$b_{i+1}-b_i = a^v_{X_{t_i}}(r) = \E(c_{i+1}-c_i \mid \cF_{t_{i+1}})$.
\end{proof}

In words, we partition the walk into excursions, each of which start and ends at $A_v(R')$, and for each excursion we count the number of visits to $v$ and subtract the expectation.

Let $I$ be the first index such that either $b_I \ge K d_v - \eps$ or $\tcov^*(B_v(R+r)) \le  t_I$. Obviously, this is a stopping time and also $b_I \le K d_v +1$ since $a^v_w(R')\le 1$, for all $w\in A_v(R')$.

\begin{claim} \label{no_visit}
$$\P(c_I=0) \ge e^{-3 (K d_v +1)} \ .$$
\end{claim}
\begin{proof}
The probability to hit $v$ between $t_i$ and $t_{i+1}$ is at most $a^v_{X_{t_i}}(R')$ and $\sum_{i=0}^I a^v_{X_{t_i}}(R')= b_I \le K d_v +1$. The claim now follows by Lemma \ref{bit_process}.
\end{proof}

\begin{claim} \label{small_covering_prob}
$$\P\big(\tcov^*(B_v(R+r)) \le t_I\big) \le \eps$$
\end{claim}
\begin{proof}
Obviously, $\tcov^*(B_v(R+r)) \ge \tcov^*(B_S(r))$ where $S=A_v(R')$. Let $a_w=a^v_w(R')\le a$ by our assumption on $R'$. Hence, by Lemma \ref{weak_concentration} and the choice of $a$ we have
$$\P\big(\sum_{w\in S} a_w \ell^w(r) < K \sum_{w\in S} a_w d_w -\eps\big) < \eps$$
which implies
$$\P\big(b_{\tcov^*(B_v(R+r))} < K d_v - \eps\big) < \eps$$
and the claim follows by the definition of $I$.
\end{proof}

\begin{claim} \label{Cauchy-Schwartz}
$$\E(c_I) \ge (K d_v - \eps)(1-\eps) \, .$$
\end{claim}
\begin{proof}
$\E(c_I)=\E(b_I)$ by Lemma \ref{martingale} and since $I$ is a stopping time.
$$\E(b_I) \ge (K d_v -\eps)\P(b_I \ge K d_v -\eps) \ge (K d_v - \eps) (1-\eps)$$
by claim \ref{small_covering_prob} and the definition of $I$.
\end{proof}

Summing it all up, the expected number of visits to $v$ before $\tcov^*(B_v(R+r))$ is at least the expected number of these visits which occur before $t_I$ plus the probability that $v$ has not been visited at all by time $t_I$ (in which case we need to visit it at least once). Lemma \ref{Cauchy-Schwartz} and Lemma \ref{no_visit} bound these from below yielding
\begin{eqnarray*}
\E(\ell^v(R+r)) &\ge& \E(c_I) + \P(c_I=0)\\
&\ge& (K d_v - \eps)(1-\eps) + e^{-3 (K d_v +1)}\\
&\ge& K d_v + e^{-3 K d_v -4}
\end{eqnarray*}
for $\eps$ small enough.
\end{proof}

\medskip
\noindent {\bf Acknowledgements:} The authors thank Ariel Yadin for useful discussions.

\medskip

\noindent {\bf Legend:}

$B_v(r)$ - the ball of radius $r$ around $v$

$A_v(r)=B_v(r)\setminus B_v(r-1)$ - the annulus of radius $r$ around $v$

$\tau_v(r)$ the hitting time of $A_v(r)$

$\ell^v_t$ the number of visits to $v$ before time $t$

$\ell^v(r)$ the number of visits to $v$ before time $\tau_v(r)$

$\tcov(S)$ the cover time of $S$

$\tcov$ the cover time of the graph

$\tcov^*(S)$ the time to cover and exit $S$

$\tcov^*$ the time to cover the graph and exit $B_{X_{\tcov}}(2r)$



\end{document}